\documentclass[12pt]{amsart}
\usepackage[latin1]{inputenc}
\usepackage{amscd, amsmath, mathrsfs, amssymb, amsthm, amsxtra, bbding, epsfig, eucal, eufrak, graphicx, latexsym, mathrsfs}
\usepackage{amsxtra,amssymb,amsthm,amsmath,amscd,mathrsfs, epsfig,eufrak,url}

\usepackage[all]{xy}
\def \A {{\mathbb A}}

\def \C {{\mathbb C}}

\def \Q {{\mathbb Q}}
\def \R {{\mathbb R}}

\def \d {\,{\rm d}}
\def\re{{\Re e\,}}

\def\le{\leqslant}
\def\leq{\leqslant}
\def\ge{\geqslant}
\def\geq{\geqslant}

\theoremstyle{plain}
\newtheorem{theorem}{Theorem}

\newtheorem{lemma}{Lemma}[section]
\newtheorem{corollary}{Corollary}

\theoremstyle{remark}
\newtheorem{remark}{Remark}

\theoremstyle{definition}

\numberwithin{equation}{section}

\newtheorem*{hypothesis}{Hypothesis H}

\begin{document}

\vskip 5mm

\title[Coefficients of automorphic $L$-functions for $GL_m$ of same signs]
{The number of coefficients of automorphic $L$-functions for $GL_m$ of same signs} 
\author{Jianya Liu \& Jie Wu}

\address{
Jianya Liu
\\
School of Mathematics
\\
Shandong University
\\
Jinan
\\
Shandong 250100
\\
China}
\email{jyliu@sdu.edu.cn}

\address{%
Jie Wu
\\
CNRS\\
Institut \'Elie Cartan de Lorraine\\
UMR 7502\\
54506 Van\-d\oe uvre-l\`es-Nancy\\
France}
\curraddr{%
Universit\'e de Lorraine\\
Institut \'Elie Cartan de Lorraine\\
UMR 7502\\
54506 Van\-d\oe uvre-l\`es-Nancy\\
France
}
\email{jie.wu@univ-lorraine.fr}

\date{\today}

\begin{abstract}
Let $\pi$  be an irreducible unitary cuspidal representation 
for $GL_m(\A_\Q)$, and let $L(s, \pi)$ be the automorphic $L$-function 
attached to $\pi$, which has a Dirichlet series expression in  the half-plane 
$\re s>1$. When $\pi$ is self-contragredient, all the coefficients in the 
Dirichlet series expression are real. In this paper we give non-trivial lower 
bounds for the number of positive and negative coefficients, respectively. 

\end{abstract}
\maketitle

\addtocounter{footnote}{1}

\section{Introduction}

Let $m\ge 2$ be an integer and 
let $\pi=\otimes \pi_p$ be an irreducible unitary cuspidal representation of 
$GL_m(\A_\Q)$.
The corresponding global $L$-function is defined by the product of local factors
\begin{equation}\label{defLspi} 
L(s, \pi)
:= \prod_{p<\infty} L_p (s, \pi_p)
\end{equation}
for $\re s>1$, where
\begin{equation}\label{defLpspip}
L_p(s, \pi_p ) := \prod_{1\le j\le m} \bigg(1-\frac{\alpha_\pi(p,j)}{p^{s}}\bigg)^{-1}.
\end{equation}
The complete $L$-function $\Phi(s,\pi)$ is defined by
\begin{equation}\label{PHIspi}
\Phi(s,\pi)=L_\infty(s,\pi_\infty)L(s,\pi),
\end{equation}
where
\begin{equation}\label{L8spi8}
L_\infty(s, \pi_\infty) := \pi^{-ms/2}
\prod_{1\le j\le m} \Gamma\bigg(\frac{s-\mu_{\pi}(j)}{2}\bigg)
\end{equation}
is the Archimedean local factor. 
Here $\{\alpha_\pi(p,j)\}_{j=1}^m \subset \C $ and $\{\mu_{\pi} (j)\}_{j=1}^m\subset \C$
are local parameters associated with $\pi_p$ and $\pi_\infty$,
respectively, according to the Langlands correspondence.
Good bounds for these local parameters are of fundamental importance 
for the study of automorphic $L$-functions.
The best known record is due to Kim \& Sarnak \cite{KimSarnak2003} ($2\le m\le 4$) 
and Luo, Rudnick \& Sarnak \cite{LuoRudnickSarnak1999} ($m\ge 5$) :
\begin{equation}\label{LRS}
|\alpha_\pi(p, j)|\leq p^{\theta_m} 
\qquad\text{and}\qquad
|\re \mu_\pi(j)| \leq \theta_m 
\end{equation}
for all primes $p$ and $1\le j\le m$, where 
\begin{equation}\label{defthetam}
\theta_2 := \frac{7}{64},
\qquad
\theta_3 := \frac{5}{14},
\qquad
\theta_4 := \frac{9}{22},
\qquad
\theta_m := \frac{1}{2}-\frac{2}{m^2+1} 
\quad
(m\ge 5).
\end{equation}
The Generalized Ramanujan Conjecture (GRC in brief) asserts that
the inequalities in \eqref{LRS} hold for all primes $p$ and $1\le j\le m$ with 
\begin{equation}\label{GRC}
\theta_m=0.
\end{equation}

Jacquet \& Shalika \cite{JacquetShalika1981} showed that
the Euler product for $L(s,\pi)$ in \eqref{defLspi} converges absolutely for $\re s>1$.
Thus, in this half-plane, we may write
\begin{equation}\label{Lspi=Dir}
L(s,\pi)=\sum_{n=1}^\infty \frac{\lambda_{\pi}(n)}{n^s},
\end{equation}
where
\begin{equation}\label{lambdapin}
\lambda_{\pi}(n)
:= \prod_{p^\nu\| n} \sum_{\nu_1+\cdots+\nu_m = \nu} 
\prod_{1\le j\le m} \alpha_{\pi}(p, j)^{\nu_j}.
\end{equation}
In particular if $m=2$ and 
$\pi$ is corresponding to cusp form, then $\lambda_{\pi}(n)$ appears as a Fourier coefficient in the Fourier expansion of $\pi$.
These coefficients are mysterious objects and an interesting question is how, 
for a fixed representation,
the coefficients $\lambda_{\pi}(n)$ are distributed.
In the case of $m=2$,
there are many results from which the distribution appears to be highly random; 
for example, the recent proof of the Sate-Tate conjecture on the Fourier coefficients of holomorphic Hecke eigenforms of Barnet-Lamb, Geraghty, Harris \& Taylor \cite{BGHT2011}.
We refer to \cite{LauLiuWu2012} for a more complete survey.

It seems natural and interesting to investigate the case of $m\ge 3$.
By a classical method of Landau \cite{Landau1915},
Qu \cite[Theorem 1.1]{Qu2010} proved that 
if $\pi$ is an irreducible unitary cuspidal representation for $GL_m(\A_\Q)$ 
such that $\lambda_{\pi}(n)$ is real for all $n\ge 1$ (for example, if $\pi$ is self-contragredient), 
then there must be infinitely many sign changes in the sequence $\{\lambda_{\pi}(n)\}_{n=1}^{\infty}$, 
i.e., there are infinitely many $n$ such that $\lambda_{\pi}(n)> 0$, 
and there are infinitely many $n$ such that $\lambda_{\pi}(n)< 0$.
In this paper, 
we would like to give a quantitative version of this result.
The key point of our method is to establish an asymptotic formula for the weighted second moment of $\lambda_{\pi}(n)$ 
defined as 
\begin{equation}\label{moments}
S_{\pi, \kappa}(x) := \sum_{n\le x} \frac{|\lambda_{\pi}(n)|^2}{d_{\kappa}(n)},
\end{equation}
where $\kappa>0$ is a constant, $\zeta(s)$ is the Riemann zeta-function and $d_{\kappa}(n)$ is the Piltz function defined by the relation
\begin{equation}\label{defdkappan}
\zeta(s)^{\kappa} = \sum_{n\ge 1} d_{\kappa}(n) n^{-s}
\quad
(\re s>1).
\end{equation}
Our result (see Theorem \ref{thm4} below) is unconditional for $2\le m\le 4$.
When $m\ge 5$, we only need a weak assumption, i.e.  
the well known Hypothesis H of Rudnick \& Sarnak \cite{RudnickSarnak1996}, instead of GRC \eqref{GRC}.
We shall see that the former is a trivial consequence of the later.
In order to state Hypothesis H, 
let us first fix some notation.
Since the Euler products \eqref{defLspi}-\eqref{defLpspip} converge absolutley 
for $\re s>1$, we can write, in this half plane,
$$
- \frac{L'}{L}(s, \pi) = \sum_{n=1}^{\infty} \frac{\Lambda(n) a_{\pi}(n)}{n^s},
$$
where $\Lambda(n)$ is the von Mangoldt function defined by 
\begin{equation}\label{Lambdan}
\Lambda(n) := \begin{cases}
\log p & \text{if $n=p^\nu$}
\\\noalign{\vskip 1mm}
0        & \text{otherwise}
\end{cases}
\end{equation}
and
\begin{equation}\label{apin}
a_{\pi}(n) := \begin{cases}
\alpha_{\pi}(p, 1)^{\nu} + \cdots + \alpha_{\pi}(p, m)^{\nu} & \text{if $n=p^\nu$}
\\\noalign{\vskip 1mm}
0 & \text{otherwise}
\end{cases}
\end{equation}
for all primes $p$ and integers $\nu\ge 1$. 
Hypothesis H of Rudnick-Sarnak states the following. 

\begin{hypothesis}
{\it For any fixed $\nu\geq 2$, 
\begin{equation}\label{H}
\sum_p \frac{|a_\pi(p^\nu)|^2(\log p)^2}{p^\nu}
<\infty.
\end{equation}} 
\end{hypothesis}

\begin{remark}\label{remark1}
(i) 
Clearly Hypothesis H is a simple consequence of GRC.
\par
(ii)
For $m=2, 3$, Hypothesis H follows from the Rankin-Selberg theory \cite{RudnickSarnak1996}. 
The $GL_4(\A_\Q)$ case was proved by Kim \cite{Kim2006} based on his proof 
of the (weak) functoriality of the exterior square $\wedge^2 \pi$ 
from a cuspidal representation $\pi$ of $GL_4(\A_\Q)$ (cf. \cite{Kim2003}). 
Beyond $GL_4(\A_\Q)$, 
the only known special cases for Hypothesis H 
are as follows :
\begin{itemize}
\item{the symmetric fourth power ${\rm sym}^4\pi$
of a cuspidal representation $\pi$ of $GL_2(\A_\Q)$, 
which is an automorphic representation of $GL_5(\A_\Q)$
(see \cite{Kim2003, Kim2006});}

\item{the automorphic representation $\Pi$ of $GL_6(\A_\Q)$ 
such that $\Pi_p\cong\wedge^2\pi_p$ if $p\not\in T$, 
where $\pi$ is a cuspidal representation of $GL_4(\A_\Q)$
and $T$ is the set of places consisting of $p=2,3$ and those $p$ at which $\pi_p$ is supercuspidal
\cite[Theorem~1]{WuYe2007};}

\item{the automorphic representation $\pi_1\boxtimes\pi_2$ of $GL_6(\A_\Q)$,
where $\pi_1$ $($resp. $\pi_2)$ 
is a cuspidal representation of $GL_2(\A_\Q)$ $($resp. $GL_3(\A_\Q))$
(cf. \cite[Theorem 2]{WuYe2007}).}
\end{itemize}
\par
(iii)
From the proof of Theorem \ref{thm1}, we can see that the result of this theorem in the case of $m\ge 5$ 
also holds under a slightly weaker assumption that
for any fixed integer $\nu\geq 2$, 
\begin{equation}\label{FH}
\sum_p \frac{|a_\pi(p^\nu)|^2\log p}{p^\nu}
<\infty.
\end{equation}
\end{remark}

\smallskip

Now we write 
\begin{equation}\label{defNpix}
{\mathscr N}_{\pi}^{\pm}(x)
:= \sum_{\substack{n\le x\\ \lambda_{\pi}(n)\gtrless\,0}} 1,
\end{equation}
and we are interested in their asymptotic behaviour as $x\to\infty$. 
It seems reasonable to conjecture 
\begin{equation}\label{ConjectureNpix}
{\mathscr N}_{\pi}^{\pm}(x)
\gg_{\pi} x
\end{equation}
for $x\ge x_0(\pi)$, where $x_0(\pi)$ is a constant depending on $\pi$.

\vskip 2mm

The principal aim of this paper is to prove the following result.

\begin{theorem}\label{thm1}
Let $\pi$ be a self-contragredient irreducible unitary cuspidal 
representation for $GL_m(\A_\Q)$, and let $\theta_m$ be as in 
(\ref{defthetam}). 
Then we have 
$$
{\mathscr N}_{\pi}^{\pm}(x)
\gg_{\pi} x^{1-2\theta_m} (\log x)^{2/m-2}
\quad
(x\ge x_0(\pi))$$
unconditionally for $2\le m\le 4$ and under Hypothesis H for 
$m\ge 5$, 
where the constant $x_0(\pi)$
and the implied constant depend only on $\pi$.
\end{theorem}

\begin{corollary}\label{cor1}
Let ${\rm sym}^4\pi$, $\Pi$ and $\pi_1\boxtimes\pi_2$ be the automorphic 
representations of $GL_5(\A_\Q)$ and $GL_6(\A_\Q)$ mentionned in 
Remark \ref{remark1} {\rm (ii)}, respectively.
Then we have 
$$
{\mathscr N}_{{\rm sym}^4\pi}^{\pm}(x)
\gg x^{2/13} (\log x)^{-8/5}
$$
and
$$
{\mathscr N}_{\Pi}^{\pm}(x), \, 
{\mathscr N}_{\pi_1\boxtimes\pi_2}^{\pm}(x)
\gg x^{4/37} (\log x)^{-5/3}
$$
unconditionally for $x\ge x_0$, where the constant $x_0$ and the implied constants depend on ${\rm sym}^4\pi$, $\Pi$ 
and $\pi_1\boxtimes\pi_2$ respectively.
\end{corollary}

\begin{corollary}\label{cor2}
Let $\pi$ be a self-contragredient irreducible unitary cuspidal 
representation for $GL_m(\A_\Q)$. 
Then the number of sign changes of the sequence $\{\lambda_{\pi}(n)\}_{n\ge 1}$ in the interval $[1, x]$ is $\gg_{\pi} \log\log x$
unconditionally for $2\le m\le 4$ and under Hypothesis H for $m\ge 5$.
\end{corollary}

When $\pi$ is an holomorphic primitive modular form $f$ of $GL_2$, 
Lau \& Wu \cite[Theorem 1]{LauWu2009} obtained the best possible lower bounds \eqref{ConjectureNpix}.
Their method is completely different from ours and it seems rather difficult to generalize it to our case.
One of difficulties is that in the general case there is no analogue of Serre's estimate \cite[page 181]{Serre1981} :
$$
|\{p\le x : \lambda_f(p)=0\}|
\ll_{f, \delta}\frac{x}{(\log x)^{1+\delta}}
$$
for $x\ge 2$ and any $\delta<\frac{1}{2}$.
Here, as indicated before, 
we shall prove Theorem \ref{thm1} by establishing an asymptotic formula of the weighted second moment \eqref{moments}.
In fact we shall obtain a rather general mean value theorem on non-negative multiplicative functions
(see Theorem \ref{thm2} in Section 2 below), and 
\eqref{moments} is just a particular case of this. 
This general result is of independent interest and may find other 
applications in the future. 

\vskip 3mm
 
\noindent 
{\bf Acknowledgements.} Liu is supported in part by the 973 program 
and NSFC grant 11031004, and both authors are supported in part 
by IRT1264 from the Ministry of Education.  
This work was done partly during the visit of the first author 
to l'Institut Elie Cartan de l'Universit\'e de Lorraine.
He wishes to thank this institute for the hospitality and support.

\vskip 8mm

\section{Mean values of multiplicative functions}

Let $f(n)$ be a non-negative multiplicative function satisfying certain growth conditions
that will be specified later. 
In order to prove Theorem \ref{thm1}, we need to establish an asymptotic formula for
\begin{equation}\label{defSfx}
S_f(x) := \sum_{n\le x} f(n).
\end{equation}
Assume that there are positive constants $A>0$, $\kappa>0$ and $\eta\in (0, \tfrac{1}{2})$ such that
\begin{eqnarray}\label{TWcondition1}
\sum_{p\le z} f(p)\log p
= \kappa z + O(z/R(z))
\quad
(z\ge 2)
\end{eqnarray}
where the summation is taken over primes $p$, and $z\mapsto R(z)$ is an increasing function satisfying certain growth conditions, and such that 
\begin{eqnarray}\label{TWcondition2}
\sum_{p, \, \nu\ge 2} \frac{f(p^\nu)}{p^{ (1-\eta)\nu}}
\le A. 
\end{eqnarray}
By the saddle point meothd, 
Tenenbaum \& Wu \cite{TenenbaumWu2003} obtained a sharp asymptotic formula 
for the mean value of $f(n)$ over friable integers :
$$
S_f(x, y) := \sum_{\substack{n\le x\\ p\mid n\Rightarrow p\le y}} f(n)
$$
in a large domain of $(x, y)$.
Taking $y=x$, their result yields an asymptotic formula for $S_f(x)$.
Unfortunately the condition \eqref{TWcondition2} is too restrictive 
for our application.
In this section, we shall establish a rather general asymptotic formula for $S_f(x)$.
On the one hand, we shall relax \eqref{TWcondition2} to \eqref{Condition2} below,
which is necessary for our application.  
On the other hand, we shall consider a general form $R(z)$ for the error term of \eqref{TWcondition1}
as in \cite{TenenbaumWu2003} such that it is more convenient for 
applications in the future.

Denote by ${\mathscr R}$ the set of all increasing functions $R\in {\mathscr C}^1((1, \infty), (1, \infty))$ satisfying the following conditions : there are two positive constants $z_0=z_0(R)$ and $\delta=\delta(R)$ such that
\par
{\rm (a)}
$z\mapsto \frac{R'(z)}{R(z)}z\log z$ is monotonic 
and $\frac{R'(z)}{R(z)}z$ is bounded for $z\ge z_0$;
\par
{\rm (b)}
we have $R(z)\gg (\log_2z)^{1+\delta}$ for $z\ge z_0$.  

\noindent 
For simplicity in the following we will write 
\begin{eqnarray}\label{def/R+}
R^+ (z) : =\frac{R'(z)}{R(z)}z\log z. 
\end{eqnarray}

\vskip 1mm

The definition of ${\mathscr R}$ may seem technical,  
but the conditions are easily verified in practice. 
In fact, most of the explicit error terms of arithmetic sums actually correspond to elements of such a class.
The following functions are, defined for $z>\text{e}$, typical 
examples of elements of ${\mathscr R}$ :
$$
(\log_2 z)^{1+\delta},
\qquad
(\log z)^{\delta}/(\log_2z)^\eta,
\qquad
\text{e}^{(\log_2z)^\delta},
\qquad
\text{e}^{(\log z)^\delta},
\qquad
z^\delta,
$$
where $\delta>0$ is a positive constant, $\eta\in \R$ is a real number
and $\log_k$ is the $k$-fold logarithmic function.

\vskip 1mm

The principal result of this section is as follows.

\begin{theorem}\label{thm2}
Let $f(n)$ be a non-negative multiplicative function satisfying the condition \eqref{TWcondition1} with some $R\in {\mathscr R}$
and the inequality
\begin{equation}\label{Condition2}
\sum_{p, \, \nu\ge 2} \frac{f(p^\nu)}{p^\nu} \log p^\nu\le A,
\end{equation}
where $A>0$ is a constant.
Then we have
\begin{equation}\label{AsymptoticSfx}
S_f(x) = C_fx(\log x)^{\kappa-1} \bigg\{1
+ O_f\bigg(\frac{(\log_2x)^2}{\log x} + {\mathcal E}(x)\bigg)\bigg\},
\end{equation}
where
\begin{equation}\label{defCf}
C_f := \frac{1}{\Gamma(\kappa)} 
\prod_p\bigg(1+\sum_{\nu\ge 1} \frac{f(p^\nu)}{p^\nu}\bigg) \bigg(1- \frac{1}{p}\bigg)^\kappa
\end{equation}
and
\begin{equation}\label{defMathcalEx}
{\mathcal E}(x) := 
\begin{cases}
\displaystyle \frac{(\log R(x))^2}{R(x)} + \int_x^\infty \!\! \frac{(\log R(z))^2}{R(z)z\log z} \d z
& \mbox{if } R^+ (z)\uparrow 1, 
\\\noalign{\vskip 2mm}
\displaystyle \frac{1}{R(x)} + \int_x^\infty \frac{1}{R(z)z\log z} \d z
& \text{otherwise.}   
\end{cases}
\end{equation}
Here and throughout, $g(z)\uparrow 1$ means that the function $g(z)$ monotonically 
increases to the limit $1$ as $z\to \infty$.  
\end{theorem}

In order to prove Theorem \ref{thm2}, we first state a simple lemma due to 
Hall \& Tenenbaum \cite[Theorem 01]{HallTenenbaum1988},
which gives an upper bound for $S_f(x)$ in term of the logarithmic mean of $f(n)$ defined by
\begin{equation}\label{defsfx}
s_f(x) := \sum_{n\le x} \frac{f(n)}{n}\cdot
\end{equation}

\begin{lemma}\label{UpperBoundSfxsfx}
Let $f(n)$ be a non-negative multiplicative function satisfying a weaker version of \eqref{TWcondition1} 
with certain constant $B>0 :$
\begin{equation}\label{UBcondition1}
\sum_{p\le z} f(p) \log p\le Bz
\quad
(z\ge 2)
\end{equation}
and the condition \eqref{Condition2}.
Then for all $x\ge 2$ we have
\begin{equation}\label{UBSfxsfx}
S_f(x)\le (A+B+1)\frac{x}{\log x}s_f(x).
\end{equation}
\end{lemma}

\begin{proof}
We write
$$
\sum_{n\le x} f(n)\log n
= \int_{1-}^x \log t \d S_f(t)
= S_f(x) \log x - \int_1^x \frac{S_f(t)}{t} \d t.
$$
We have trivially
$$
\int_1^x \frac{S_f(t)}{t} \d t
\le \int_1^x s_f(t) \d t
\le xs_f(x).
$$
With the help of \eqref{UBcondition1} and \eqref{Condition2}, we can deduce that
\begin{align*}
\sum_{n\le x} f(n)\log n
& = \sum_{\substack{p^\nu m\le x\\ p\nmid m}} f(m) f(p^\nu) \log p^\nu
\\
& \le \sum_{m\le x} f(m) \sum_{p\le x/m} f(p) \log p
+ x\sum_{m\le x} \frac{f(m)}{m} \sum_{\substack{p^{\nu}\le x/m\\ \nu\ge 2}} \frac{f(p^\nu)}{p^{\nu}} \log p^\nu
\\
& \le Bxs_f(x) + Axs_f(x).
\end{align*}

Now the required inequality \eqref{UBSfxsfx} follows from these relations.
\end{proof}

Next we shall evaluate $s_f(x)$.
This was studied by Halberstam  in an unpublished manuscript, 
who obtained an asymptotic formula for $s_f(x)$ 
by using the technique of \cite[Chapter 5]{HR1974}.  
A complete proof of Halberstam's result can be found in Song's paper \cite[Theorem A]{Song2001}.
Theorem \ref{thm3} below can be regarded as a simple generalization of Halberstam's result
(with the choice of $R(z) = (\log z)^{\delta}$ with $\delta\in (0, 1)$).

\begin{theorem}\label{thm3}
Let $R\in {\mathscr R}$ and let $f(n)$ be a non-negative multiplicative function satisfying, for certain positive constant $\kappa>0$,
\begin{equation}\label{Hcondition1}
\sum_{p\le z} \frac{f(p)}{p}\log p = \kappa \log z + O\bigg(\frac{\log z}{R(z)} + \log_2z\bigg)
\end{equation}
and \eqref{Condition2}.
Then for all $x\ge 2$ we have
\begin{equation}\label{Asymptotic.sfx}
s_f(x) = c_f(\log x)^\kappa\bigg\{1 + O_f\bigg(\frac{(\log_2x)^2}{\log x} + {\mathfrak E}(x)\bigg)\bigg\},
\end{equation}
where 
\begin{equation}\label{defcf}
c_f := \frac{1}{\Gamma(\kappa+1)} 
\prod_p\bigg(1+\sum_{\nu\ge 1} \frac{f(p^\nu)}{p^\nu}\bigg) \bigg(1- \frac{1}{p}\bigg)^\kappa
\end{equation}
and
\begin{equation}\label{defMathfrakEx}
{\mathfrak E}(x) := 
\begin{cases}
\displaystyle \frac{\log R(x)}{R(x)} + \int_x^\infty \!\! \frac{\log R(z)}{R(z)z\log z} \d z
& \text{if $R^+ (z)\uparrow 1$},
\\\noalign{\vskip 2mm}
\displaystyle \frac{1}{R(x)} + \int_x^\infty \frac{1}{R(z)z\log z} \d z
& \text{otherwise.}
\end{cases}
\end{equation}
\end{theorem}

The assumption $R(z)\gg (\log_2z)^{1+\delta} \; (z\ge z_0)$ implies that the last integral is convergent. 

From Lemma \ref{UpperBoundSfxsfx} and Theorem \ref{thm3}, we easily derive an upper bound for $S_f(x)$,
which will play a key role in the proof of Theorem \ref{thm2}.

\begin{corollary}\label{cor}
If $f$ is as in Theorem \ref{thm2}, then
\begin{equation}\label{UBSfx}
S_f(x)\ll x(\log x)^{\kappa-1}
\end{equation}
for all $x\ge 2$.
\end{corollary}

Before proving Theorem \ref{thm3}, Corollary \ref{cor} and Theorem \ref{thm2}, 
we need to establish some simple estimates about $R(z)$.

\begin{lemma}\label{lemmaRz}
Let $R\in {\mathscr R}$.
There is a constant $C=C(R)>0$ such that for $z>1$ the following 
four estimates hold: 
\begin{align}
\frac{z}{R(z)}
& \le C\bigg(1 + \int_1^z \frac{\d t}{R(t)}\bigg),
\label{Rcondition1}
\\
\frac{\log t}{R(t)}
& \le C\bigg(1 + \frac{\log z}{R(z)}\bigg)
\quad
(1\le t\le z), 
\label{Rcondition2}
\\
\int_{2}^z \frac{\d t}{tR(t)}
& \le C \bigg(\log_2z + (\log z)\frac{\log R(z)}{R(z)}\bigg),
\label{Rcondition3}
\\
\int_{2}^z \frac{R'(t)\log t}{R(t)^2} \d t
& \le C \bigg(\log_2z + (\log z)\frac{\log R(z)}{R(z)}\bigg). 
\label{Rcondition4}
\end{align}
The factor $\log R(z)$ in \eqref{Rcondition3} and \eqref{Rcondition4}  
appears only if $R^+ (z)\uparrow 1.$ 
\end{lemma}

\begin{proof}
Put 
$$
g(z) := \frac{z}{R(z)} - \int_1^z \frac{\d t}{R(t)}\cdot
$$ 
We have 
$$
g'(z) = - z\frac{R'(z)}{R(z)^2}\le 0
\quad
(z\ge 3).
$$
Thus $g(z)\le g(3)$ for $z\ge 3$.
Clearly this implies \eqref{Rcondition1}.

Put $h(z) := (\log z)/R(z)$. We have
$$
h'(z) = \frac{1-R^+ (z)}{zR(z)}. 
$$
Since $z\mapsto R^+ (z)$ is monotonic on $[z_0, \infty)$, 
there is a constant $z_1=z_1(R)$ such that $h'(z)$ is of constant sign for $z\ge z_1$. 
Thus $h(z)$ is monotonic on $[z_1, \infty)$,
which implies that
$$
h(t)\le h(z_1) + h(z)
\qquad
(z_1\le t\le z).
$$
From this we deduce that there is a constant $C= C(R)>0$ such that \eqref{Rcondition2} holds.

The inequality \eqref{Rcondition3} is Lemma 3.3 of \cite{TenenbaumWu2003}.

Finally we prove \eqref{Rcondition4}.
Since $z\mapsto R^+ (z)$ is positive and monotonic, it tends to a limit.
If this limit is finite, there are two constants $t_0$ and $M$ such that $R^+ (t)\le M$ for $t\ge t_0$.
Thus for $z\ge t_0$, 
$$
\int_{2}^z \frac{R'(t)\log t}{R(t)^2} \d t
\le \int_{2}^{t_0} \frac{R'(t)\log t}{R(t)^2} \d t
+ M\int_{2}^z \frac{\d t}{tR(t)} 
$$
and the desired inequality follows from \eqref{Rcondition3}.
If the limit is infinite, then for any constant $D>2$ there is a constant $t_0$ 
such that $R^+ (t)\ge D$ for $t\ge t_0$.
This implies that $R(t)\gg (\log t)^D$  for $t\ge t_0$.
In view of the hypothesis $\frac{R'(t)}{R(t)}t\ll 1$, we have
$$
\int_{2}^z \frac{R'(t)\log t}{R(t)^2} \d t
\ll \int_{2}^{z} \frac{\d t}{t(\log t)^{D-1}} \d t
\ll 1. 
$$
This completes the proof.
\end{proof}

Now we are ready to prove Theorem \ref{thm3}, Corollary \ref{cor} and Theorem \ref{thm2}.

\subsection{Proof of Theorem \ref{thm3}}
The proof is similar to that of Theorem A of \cite{Song2001},
which is a combination of Halberstam's proof and Hildebrand's identity \cite{Hildebrand1986}.
The only difference comes from our general form $R(z)$ for the error term in \eqref{Hcondition1}.

A simple partial integration gives 
\begin{equation}\label{B1}
\sum_{n\le x} \frac{f(n)}{n} \log n
= \int_{1-}^x \log t \d s_f(t)
= s_f(x) \log x - \int_1^x \frac{s_f(t)}{t} \d t.
\end{equation}
On the other hand, the multiplicativity of $f(n)$ allows us to write
\begin{equation}\label{B2}
\begin{aligned}
\sum_{n\le x} \frac{f(n)}{n} \log n
& = \sum_{\substack{mp^\nu\le x\\ p\nmid m}} \frac{f(m)}{m} \frac{f(p^\nu)}{p^\nu} \log p^\nu
\\
& = \sum_{mp\le x} \frac{f(m)}{m} \frac{f(p)}{p} \log p - E_1 + E_2,
\end{aligned}
\end{equation}
where
\begin{align*}
E_1
& := \sum_{\substack{mp\le x\\ p\mid m}} \frac{f(m)}{m} \frac{f(p)}{p} \log p,
\\
E_2
& := \sum_{\substack{mp^\nu\le x\\ p\nmid m, \,\nu\ge 2}} \frac{f(m)}{m} \frac{f(p^\nu)}{p^\nu} \log p^\nu.
\end{align*}
Clearly the non-negativity of $f(n)$ and the hypothesis \eqref{Condition2} imply
\begin{equation}\label{B3}
E_2
\le \sum_{p, \nu\ge 2} \frac{f(p^\nu)}{p^\nu} \log p^\nu \sum_{m\le x/p^\nu} \frac{f(m)}{m}
\le As_f(x).
\end{equation}
For $p\mid m$, we write $m=p^\nu\ell$ with $p\nmid \ell$. Then $f(m) = f(p^\nu)f(\ell)$ and
$$
E_1
\le s_f(x)
\sum_{p^{\nu+1}\le x, \, \nu\ge 1} \frac{f(p^\nu)}{p^\nu} \frac{f(p)}{p} \log p.
$$
Noticing that the hypothesis \eqref{Hcondition1} implies 
$$
\frac{f(p)}{p}\log p
\ll \log_2p+\frac{\log p}{R(p)}
\qquad\text{and}\qquad
q(t)
:= \sum_{p\le t} \frac{f(p)}{p}\log p\asymp \log t,
$$
we can deduce, by \eqref{Condition2} and \eqref{Rcondition4},
that
\begin{align*}
\sum_{p\le \sqrt{x}} \frac{f(p)^2}{p^2} (\log p)
& \ll \sum_{p\le x} \frac{f(p)}{p}\log_2p + \sum_{p\le x} \frac{f(p)}{R(p)p}\log p
\\
& \ll (\log_2x)^2 + \frac{\log x}{R(x)} + \int_2^{x} \frac{R'(t)\log t}{R(t)^2} \d t
\\
& \ll (\log_2x)^2 + (\log x)\frac{\log R(x)}{R(x)},
\end{align*}
where the factor $\log R(x)$ appears 
only if $R^+(z)\uparrow 1$. 
Similarly
$$
\sum_{p^{\nu+1}\le x, \, \nu\ge 2} \frac{f(p^\nu)}{p^\nu} \frac{f(p)}{p} \log p
\ll \sum_{p, \, \nu\ge 2} \frac{f(p^\nu)}{p^\nu} \log p
\ll A.
$$
Combining these estimates, we obtain
\begin{equation}\label{B4}
E_1
\ll s_f(x)\bigg((\log_2x)^2 + (\log x)\frac{\log R(x)}{R(x)}\bigg).
\end{equation}

Further we can apply the hypothesis \eqref{Hcondition1} and \eqref{Rcondition2} of Lemma \ref{lemmaRz} to write
\begin{equation}\label{B5}
\begin{aligned}
\sum_{mp\le x} \frac{f(m)}{m} \frac{f(p)}{p} \log p
& = \sum_{m\le x} \frac{f(m)}{m} \bigg\{\kappa \log\frac{x}{m} 
+ O\bigg(\log_2\Big(\frac{3x}{m}\Big) +\frac{\log(x/m)}{R(x/m)}\bigg)\bigg\}
\\
& = \kappa \int_1^x \frac{s_f(t)}{t} \d t
+ O\bigg(s_f(x)\bigg(\log_2x + \frac{\log x}{R(x)}\bigg)\bigg).
\end{aligned}
\end{equation}
Combining \eqref{B1}, \eqref{B2}, \eqref{B3}, \eqref{B4} and \eqref{B5}, we find that
$$
s_f(x)
= \frac{\kappa + 1}{\log x} \int_1^x \frac{s_f(t)}{t} \d t + s_f(x)\varepsilon(x),
$$
where
$$
|\varepsilon(x)| 
\le C\bigg(\frac{(\log_2x)^2}{\log x}   
+ \frac{\log R(x)}{R(x)}\bigg)
\le \frac{1}{2}
\qquad
(x\ge x_0).
$$
Here the factor $\log R(x)$ 
appears only if $R^+ (z)\uparrow 1$. 
From this we can derive
\begin{equation}\label{B6}
s_f(x) = \frac{1}{1-\varepsilon(x)} \frac{\kappa+1}{\log x}\int_1^x \frac{s_f(t)}{t} \d t
\qquad
(x\ge x_0).
\end{equation}

Define
$$
\varepsilon_0(z) := \log\bigg(\frac{\kappa+1}{(\log z)^{\kappa+1}}\int_1^z \frac{s_f(t)}{t} \d t\bigg).
$$
Then 
$$
\varepsilon_0'(z)
= \frac{\kappa+1}{z \log z} \frac{\varepsilon(z)}{1-\varepsilon(z)}
\ll \frac{1}{z \log z} \bigg(\frac{(\log_2z)^2}{\log z} + \frac{\log R(z)}{R(z)}\bigg)
$$
and
$$
\int_{x}^{\infty} \varepsilon_0'(z) \d z
\ll \frac{(\log_2x)^2}{\log x} + {\mathfrak E}(x) 
$$
with ${\mathfrak E}(x)$ as in (\ref{defMathfrakEx}). 
Writting $c_f := \exp\big(\int_{1}^{\infty} \varepsilon_0'(z) \d z\big)$, we have
\begin{align*}
\varepsilon_0(x)
& = \int_{1}^{\infty} \varepsilon_0'(z) \d z - \int_{x}^{\infty} \varepsilon_0'(z) \d z
\\
& = \log c_f + O\bigg(\frac{(\log_2x)^2}{\log x} + {\mathfrak E}(x)\bigg)
\end{align*}
for $x\ge x_0$ and
$$
\frac{\kappa+1}{(\log z)^{\kappa+1}}\int_1^z \frac{s_f(t)}{t} \d t
= c_f\bigg\{1+O\bigg(\frac{(\log_2x)^2}{\log x} + {\mathfrak E}(x)\bigg)\bigg\}.
$$
Inserting this into \eqref{B6} yields the required result.
\hfill
$\square$

\subsection{Proof of Corollary \ref{cor}}
By the hypothesis \eqref{TWcondition1} and \eqref{Rcondition3}, 
a simple partial integration gives 
\begin{align*}
\sum_{p\le z} \frac{f(p)}{p}\log p
& = \kappa\log z + O\bigg(1 + \int_{2}^z \frac{\d t}{tR(t)}\bigg)
\\
& = \kappa\log z + O\bigg(\frac{\log z}{R_0(z)} + \log_2z\bigg),
\end{align*}
where 
\begin{equation}\label{defR0z}
R_0(z) := \begin{cases}
R(z)/\log R(z) & \text{if $R^+ (z)\uparrow 1$,} 
\\\noalign{\vskip 1mm}
R(z) & \text{otherwise}.
\end{cases}
\end{equation}
An elementary calculation shows that
$$
\frac{R_0'(z)}{R_0(z)}
= \bigg(1 - \frac{1}{\log R(z)}\bigg)\frac{R'(z)}{R(z)}\cdot
$$
Thus it is easy to see that $R_0(z)\in {\mathscr R}$, 
and that if $R^+ (z)\uparrow 1$ as $z\to\infty$ then  
so does the function $\frac{R_0'(z)}{R_0(z)}z\log z$. 
This shows that $f$ also satisfies \eqref{Hcondition1} with the same 
$\kappa$ but $R_0(z)$ in the place of $R(z)$. 
Thus Theorem \ref{thm3} is applicable to give $s_f(x)\ll (\log x)^\kappa$.
The desired upper bound follows immediately 
from Lemma \ref{UpperBoundSfxsfx}.
\hfill
$\square$

\subsection{Proof of Theorem \ref{thm2}}
Integration by parts gives 
\begin{equation}\label{A1}
\sum_{n\le x} f(n) \log n
= \int_{1-}^x \log t \d S_f(t)
= S_f(x) \log x + O\big(x(\log x)^{\kappa-1}\big),
\end{equation}
where we have used the fact that \eqref{UBSfx} implies 
$$
\int_1^x \frac{S_f(t)}{t} \d t
\ll x(\log x)^{\kappa-1}.
$$
On the other hand, similar to \eqref{B2} and \eqref{B3}, we can write
\begin{equation}\label{A2}
\sum_{n\le x} f(n) \log n
= \sum_{mp\le x} f(m) f(p) \log p - E_1 + O\big(x(\log x)^{\kappa-1}\big),
\end{equation}
where
$$
E_1
:= \sum_{m\le x} f(m) \sum_{p\le x/m, \, p\mid m}  f(p) \log p.
$$

As in the proof of Theorem \ref{thm3}, we have
\begin{equation}\label{A3}
\begin{aligned}
E_1 
& = \sum_{\ell\le x} f(\ell) f(p^\nu) \sum_{p^{\nu+1}\le x/\ell, \, p\nmid \ell}  f(p) \log p
\\
& \ll x(\log x)^{\kappa-1} \sum_{p^{\nu+1}\le x, \, \nu\ge 1} \frac{f(p^\nu) f(p)}{p^{\nu+1}} (\log p)
\\
& \ll x(\log x)^{\kappa-1} \sum_{p\le \sqrt{x}} \frac{f(p)^2}{p^2} (\log p) + x(\log x)^{\kappa-1},
\end{aligned}
\end{equation}
where we have used the hypothesis \eqref{TWcondition1} and \eqref{Condition2} to show
$$
\sum_{\substack{p^{\nu+1}\le x\\ \nu\ge 2}} \frac{f(p^\nu) f(p)}{p^{\nu+1}} (\log p)
\ll \sum_{p, \, \nu\ge 2} \frac{f(p^\nu)}{p^{\nu}} 
\ll 1.
$$
Next we estimate the last sum of \eqref{A3}.
First we observe that the hypothesis \eqref{TWcondition1} implies $f(p)\log p\ll p/R(p)$.
By using the same hypothesis in the form
$$
Q(t)
:= \sum_{p\le t} f(p)\log p\ll t
$$
and the assumption that $\frac{R'(t)}{R(t)}t\ll 1$ and $R(t)\gg (\log_2t)^{1+\delta}$, we can deduce 
\begin{equation}\label{A4}
\begin{aligned}
\sum_{p\le \sqrt{x}} \frac{f(p)^2}{p^2} (\log p)
& \ll \sum_{p\le x} \frac{f(p)}{R(p)p}
= \int_{2-}^{x} \frac{\d Q(t)}{R(t)t\log t} 
\\
& = \frac{Q(x)}{R(x)x\log x}  
- \int_2^{x} \bigg(1 + \frac{1}{\log t} + \frac{R'(t)}{R(t)}t\bigg) \frac{Q(t)}{R(t)t^2\log t} \d t
\\
& \ll \frac{1}{R(x)\log x} + \int_2^{x} \frac{\d t}{R(t)t\log t} 
\\\noalign{\vskip 3mm}
& \ll 1.
\end{aligned}
\end{equation}

It remains to evaluate the double sum on the right-hand side of \eqref{A2}.
By the hypothesis \eqref{TWcondition1}, we can write
\begin{equation}\label{A5}
\sum_{mp\le x} f(m) f(p) \log p
= \kappa x s_f(x) + O\bigg(\sum_{m\le x} f(m)\frac{x/m}{R(x/m)}\bigg).
\end{equation}
By virtue of \eqref{Rcondition1}, \eqref{UBSfx} and \eqref{Rcondition3}, we derive
\begin{equation}\label{A6}
\begin{aligned}
\sum_{m\le x} f(m)\frac{x/m}{R(x/m)}
& \ll \sum_{m\le x} f(m) \bigg(1+\int_1^{x/m} \frac{\d t}{R(t)}\bigg)
\\
& \ll S_f(x) + \int_{1}^x \frac{S_f(x/t)}{R(t)} \d t
\\
& \ll x(\log x)^{\kappa-1} \bigg(1 + \int_{1}^x \frac{\d t}{tR(t)}\bigg) 
\\
& \ll x(\log x)^{\kappa-1}\bigg(\log_2x + (\log x)\frac{\log R(x)}{R(x)}\bigg).
\end{aligned}
\end{equation}

Combining \eqref{A2}, \eqref{A3}, \eqref{A4}, \eqref{A5}, \eqref{A6} with \eqref{A1}, we find
$$
S_f(x)
= \kappa x \frac{s_f(x)}{\log x} 
+ O\bigg(x(\log x)^{\kappa-1}\bigg(\frac{\log_2x}{\log x} + \frac{\log R(x)}{R(x)}\bigg)\bigg),
$$
where the factor $\log R(x)$ appears only if $R^+ (z)\uparrow 1$. 
According to the proof of Theorem \ref{thm3},
$f$ also satisfies \eqref{Hcondition1} with the same 
$\kappa$ but $R_0(z)$ defined as in \eqref{defR0z} in the place of $R(z)$. 
Now the desired result is a consequence of Theorem \ref{thm3}.
This completes the proof of Theorem \ref{thm2}.
\hfill
$\square$

\vskip 8mm

\section{Prime number theorem for automorphic $L$-functions}\label{GLm}

In this section we cite prime number theorem for automorphic $L$-functions.
The first such result is due to Liu, Wang \& Ye \cite[Corollary 1.2]{LiuWangYe2005}.

\begin{lemma}\label{PNTLWYLambda}
Let $\pi$ be a self-contragredient irreducible unitary cuspidal representation for $GL_m(\A_\Q)$ with $m\ge 2$.
Then there is a positive constant $c=c(\pi)$ such that
\begin{align}
\sum_{n\leq x} \Lambda(n) |a_{\pi}(n)|^2
& = x + O_{\pi}\big(x \, {\rm e}^{-c\sqrt{\log x}}\big)
\label{Lambdaapin2}
\\
\sum_{n\leq x} \Lambda(n) a_{\pi}(n)
& \ll_{\pi} x \, {\rm e}^{-c\sqrt{\log x}}
\label{Lambdaapin}
\end{align}
hold for all $x\ge 2$,
where the implied constants depend only on $\pi$.
\end{lemma}

For our purpose, it is necessary to remove the contribution of powers of primes.
The next lemma is Theorem 3 of \cite{WuYe2007},
which will play a key role in the proof of Theorem \ref{thm1}.

\begin{lemma}\label{PNTWY}
Let $\pi$ be a self-contragredient irreducible unitary cuspidal representation for $GL_m(\A_\Q)$.
Then there is a constant $c=c(\pi)>0$ such that
\begin{equation}\label{PNTaplogp}
\sum_{p\leq x} |a_\pi(p)|^2\log p
= x + O_{\pi}\big(x \, {\rm e}^{-c\sqrt{\log x}}\big)
\end{equation}
hold unconditionally for all $2\le m\le 4$ and under Hypothesis H for $m\ge 5$,
where the implied constant depends on $\pi$ only.
\end{lemma}

\vskip 8mm

\section{Proof of Theorem~\ref{thm1}}

\smallskip

We shall prove Theorem~\ref{thm1} at the end of this section, 
for which we first establish some preliminary results.  

\begin{theorem}\label{thm4}
Let $\pi$ be a self-contragredient irreducible unitary cuspidal representation for $GL_m(\A_\Q)$
and let $d_{\kappa}(n)$ be the Piltz function defined as in \eqref{defdkappan}.
Then there is a positive constant $C_{\kappa}(\pi)$ such that 
\begin{equation}\label{2ndWeightedMoment}
\sum_{n\le x} \frac{\lambda_{\pi}(n)^2}{d_{\kappa}(n)}
= C_{\kappa}(\pi) x(\log x)^{1/\kappa-1}
\bigg\{1 + O_{\pi, \kappa}\bigg(\frac{\log_2x}{\log x}\bigg)\bigg\}
\end{equation}
unconditionally for $2\le m\le 4$ and under Hypothesis H for $m\ge 5$, 
where the implied constant depends only on $\pi$ and $\kappa$.
In particular we have
\begin{equation}\label{2ndMoment}
\sum_{n\le x} \lambda_{\pi}(n)^2
= C_1(\pi) x \bigg\{1 + O_{\pi}\bigg(\frac{\log_2x}{\log x}\bigg)\bigg\}
\end{equation}
unconditionally for $2\le m\le 4$ and under Hypothesis H for $m\ge 5$, 
where the implied constant depends only on $\pi$.
\end{theorem}

\begin{proof}
We shall apply Theorem \ref{thm2} to prove \eqref{2ndWeightedMoment}.
For this, it is sufficient to verify that the non-negative multiplicative function 
$n\mapsto \lambda_{\pi}(n)^2/d_{\kappa}(n)$ satisfies the condition \eqref{TWcondition1} 
with $\kappa$ replaced by $1/\kappa$ 
and $R(z) = \text{e}^{c\sqrt{\log z}}$ (where $c$ is a positive constant 
depending on $\pi$), and the condition \eqref{Condition2} 
unconditionally for $2\le m\le 4$ and under Hypothesis H for $m\ge 5$. 

Since $d_{\kappa}(p)=\kappa$ and $\lambda_{\pi}(p)=a_{\pi}(p)$ for all primes $p$ 
(which follows from \eqref{lambdapin} and \eqref{apin} immediately), 
Lemma \ref{PNTWY} shows that
the condition \eqref{TWcondition1} is satisfied by this non-negative multiplicative function 
with the parameters indicated
unconditionally for $2\le m\le 4$ and under the Hypothesis H for $m\ge 5$.

Next we verify the condition \eqref{Condition2} : the inequality
\begin{equation}\label{(1.8)}
\sum_{\nu\ge 2} \sum_{p}
\frac{|\lambda_{\pi}(p^\nu)|^2}{d_{\kappa}(p^{\nu})p^{\nu}} \log p^\nu
\ll_{\pi, \kappa} 1
\end{equation} 
holds unconditionally for $2\le m\le 4$ and under Hypothesis H for $m\ge 5$.

Put $\eta_m := \tfrac{1}{4}(1-2\theta_m)$,
where $\theta_m$ is given by \eqref{defthetam}.
In view of \eqref{lambdapin}, \eqref{LRS} and the estimates
$$
d_{\kappa}(p^{\nu})
= \frac{1}{\nu !} \prod_{0\le j<\nu} (\kappa+j)
\ge \frac{\kappa}{\nu},
\qquad
d_m(p^{\nu}) \log p^{\nu}
\ll_m p^{\nu\eta_m},
$$
we see that 
\begin{equation}\label{UBlambdapin}
\frac{|\lambda_{\pi}(p^{\nu})|^2}{d_{\kappa}(p^{\nu})} \log p^\nu
\ll_{\kappa, m} p^{\nu 2\theta_m} d_m(p^\nu)^2 (\log p^\nu)^2
\ll_m p^{\nu 2(\theta_m+\eta_m)}
\end{equation}
for all primes $p$ and integers $\nu\ge 1$.
From this we deduce that 
\begin{equation}\label{largenu}
\begin{aligned}
\sum_{\nu\ge [1/(2\eta_m)]+2} \sum_{p} 
\frac{|\lambda_{\pi}(p^\nu)|^2}{d_{\kappa}(p^{\nu})p^{\nu}} \log p^\nu
& \ll \sum_{p} \sum_{\nu\ge [1/(2\eta_m)]+2}
\frac{1}{p^{\nu(1-2\theta_m-2\eta_m)}}
\\
& \ll \sum_{p} \sum_{\nu\ge [1/(2\eta_m)]+2}
\frac{1}{p^{\nu 2\eta_m}}
\\
& \ll \sum_{p} \frac{1}{p^{1+2\eta_m}}
\\\noalign{\vskip 1mm}
& \ll 1.
\end{aligned}
\end{equation}

It remains to estimate the contribution of $2\le \nu\le [1/(2\eta_m)]+2$ to the double sums of \eqref{(1.8)}.
To this end,  we apply the formula \cite[(5.12)]{Qu2010} 
$$
\nu \lambda_{\pi}(p^\nu)
= \sum_{1\le j\le \nu} a_{\pi}(p^j) \lambda_{\pi}(p^{\nu-j})
$$
and the Cauchy inequality to write
$$
|\lambda_{\pi}(p^\nu)|^2
\le \sum_{1\le j\le \nu} |a_{\pi}(p^j) \lambda_{\pi}(p^{\nu-j})|^2.
$$
From this and \eqref{UBlambdapin}, we derive, for $2\le \nu\le [1/(2\eta_m)]+2$,
\begin{equation}\label{lambdapnuapnu}
\begin{aligned}
\sum_{p} 
\frac{|\lambda_{\pi}(p^\nu)|^2}{d_{\kappa}(p^{\nu})p^{\nu}} \log p^\nu
& \ll_{m, \kappa} \sum_{1\le j\le \nu}\sum_{p} 
\frac{|a_{\pi}(p^j)\lambda_{\pi}(p^{\nu-j})|^2}{p^{\nu}} \log p
\\
& \ll_{m, \kappa} \sum_{1\le j\le \nu}\sum_{p} 
\frac{|a_{\pi}(p^j)|^2}{p^{j2(\theta_m+\eta_m)+\nu(1-2\theta_m-2\eta_m)}} \log p
\\
& \ll_{m, \kappa} \sum_{p} \frac{|a_{\pi}(p)|^2}{p^{1+2\eta_m}} \log p
+ \sum_{2\le j\le \nu}\sum_{p} \frac{|a_{\pi}(p^j)|^2}{p^{j}}  \log p.
\end{aligned}
\end{equation}
According to \cite[(2.2)]{WuYe2007}, for any $\varepsilon>0$ we have 
$$
\sum_p \frac{|a_{\pi}(p)|^2}{p^{1+\varepsilon}} \log p\ll_{\pi, \varepsilon} 1,
$$
which implies that
\begin{equation}\label{firstsum}
\sum_{p} \frac{|a_{\pi}(p)|^2}{p^{1+2\eta_m}} \log p\ll_{\pi} 1.
\end{equation}

When $m\ge 5$, it is clear that Hypothesis H implies
\begin{equation}\label{secondsumLargem}
\sum_{2\le j\le \nu}\sum_{p} \frac{|a_{\pi}(p^j)|^2}{p^{j}}  \log p
\ll_m \sum_{p} \frac{|a_{\pi}(p^j)|^2}{p^{j}}  \log p
\ll_\pi 1
\end{equation}
for $2\le \nu\le [1/(2\eta_m)]+2$. 

When $m\in \{2, 3, 4\}$,
Lemma 3.1(i) of \cite{WuYe2007} implies that there is a constant $\delta_m>0$ such that
\begin{align*}
\sum_{p\le x^{1/j}} |a_\pi(p^j)|^2
& \le \sum_{p^j\le x, \, j\ge 2} (\log p) |a_\pi(p^j)|^2
\\\noalign{\vskip 1mm}
& \ll x^{1-\delta_m}
\ll (x^{1/j})^{j(1-\delta_m)}
\end{align*}
for each fixed $j\ge 2$ and all $x\ge 2$.
Hence for $2\le \nu\le [1/(2\eta_m)]+2$, a simple integration by parts yields
\begin{equation}\label{secondsumSmallm}
\begin{aligned}
\sum_{2\le j\le \nu}\sum_{p} \frac{|a_{\pi}(p^j)|^2}{p^{j}}  \log p
& \ll_m \sum_{p} \frac{|a_{\pi}(p^j)|^2}{p^{j(1-\delta_m/2)}}
\\
& = \int_{2-}^{\infty} t^{-j(1-\delta_m/2)} \d \bigg(\sum_{p\le t} |a_{\pi}(p^j)|^2\bigg)
\\
& \ll t^{-j\delta_m/2}\Big|_{2}^{\infty} 
+ \int_{2}^{\infty} t^{-1-j\delta_m/2} \d t
\\\noalign{\vskip 2mm}
& \ll 1.
\end{aligned}
\end{equation}

Inserting \eqref{firstsum}, \eqref{secondsumLargem} and \eqref{secondsumSmallm} into \eqref{lambdapnuapnu}, 
we get
\begin{equation}\label{boundednu}
\sum_{2\le \nu\le [1/(2\eta_m)]+2} \sum_{p} 
\frac{|\lambda_{\pi}(p^\nu)|^2}{d_{\kappa}(p^{\nu})p^{\nu}} \log p^\nu
\ll_{m, \kappa} 1
\end{equation}
unconditionally for $2\le m\le 4$ and under Hypothesis H for $m\ge 5$.

Now \eqref{(1.8)} follows from \eqref{largenu} and \eqref{boundednu}.
\end{proof}

\begin{lemma}\label{FirstMomentLemma*}
Let $\pi$ be a self-contragredient irreducible unitary cuspidal representation for $GL_m(\A_\Q)$.
Then there is a positive constant $x_0(\pi)$ such that
\begin{equation}\label{LBS1*}
\sum_{n\le x} |\lambda_{\pi}(n)| \Big(\log\frac{x}{n}\Big)^{[m/2]+1}
\gg_{\pi} x^{1-\theta_m} (\log x)^{1/m-1}
\quad
(x\ge x_0(\pi))
\end{equation}
unconditionally for $2\le m\le 4$ and under Hypothesis H for $m\ge 5$, respectively, 
where the implied constant depends only on $\pi$.
\end{lemma}

\begin{proof}
By the first inequality of \eqref{LRS} and the multiplicativity of $\lambda_{\pi}(n)$, 
we have
$$
|\lambda_{\pi}(n)|\le n^{\theta_m}d_m(n),
$$ 
which implies that
$$
|\lambda_{\pi}(n)|
\ge \frac{\lambda_{\pi}(n)^2}{n^{\theta_m}d_m(n)} 
$$
for all $n\ge 1$.
From this, we deduce that 
\begin{align*}
\sum_{n\le x} |\lambda_{\pi}(n)| \Big(\log\frac{x}{n}\Big)^{[m/2]+1}
& \ge \sum_{n\le x} \frac{\lambda_{\pi}(n)^2}{n^{\theta_m}d_m(n)} 
\Big(\log\frac{x}{n}\Big)^{[m/2]+1}
\\
& \ge (\log 2)^{[m/2]+1} x^{-\theta_m} \sum_{n\le x/2} \frac{\lambda_{\pi}(n)^2}{d_m(n)}\cdot
\end{align*}
The desired result is an immediate consequences of 
\eqref{2ndWeightedMoment} with $\kappa=m$.
\end{proof}

\begin{lemma}\label{FirstMomentLemma}
Let $\pi$ be a self-contragredient irreducible unitary cuspidal representation for $GL_m(\A_\Q)$ with $m\ge 2$.
Then for any $\varepsilon>0$ we have
\begin{equation}\label{UBS1}
\sum_{n\le x} \lambda_{\pi}(n) \Big(\log\frac{x}{n}\Big)^{[m/2]+1}
\ll_{\pi, \varepsilon} x^{\varepsilon}
\end{equation}
for all $x\ge 2$,
where the implied constant depends on $\pi$ and $\varepsilon$.
\end{lemma}

\begin{proof}
We apply Perron's formula to write
$$
\sum_{n\le x} \lambda_{\pi}(n) \Big(\log\frac{x}{n}\Big)^{[m/2]+1} 
= \frac{1}{2\pi \text{i}}\int_{2-\text{i}\infty}^{2+\text{i}\infty}L(s,\pi)\frac{x^s}{s^{[m/2]+2}}{\rm d} s.
$$
Moving the contour to the vertical line $\re s=\varepsilon$ with $\varepsilon$ being an arbitrarily small
positive constant, 
and applying Harcos' convexity bound for $L(s,\pi)$ (see \cite{Harcos2004}) :
$$
L(\sigma+\text{i}t, \pi)\ll_{\pi, \varepsilon} (|t|+1)^{\max\{m(1-\sigma)/2, 0\}+\varepsilon}
\quad
(\sigma, \, t\in \R),
$$
we obtain
\begin{align*}
\sum_{n\le x} \lambda_{\pi}(n) \Big(\log\frac{x}{n}\Big)^{[m/2]+1}
& = \frac{1}{2\pi \text{i}}\int_{\varepsilon-\text{i}\infty}^{\varepsilon+\text{i}\infty}
L(s,\pi)\frac{x^s}{s^{[m/2]+2}} {\rm d} s 
\\
& \ll_{\pi, \varepsilon} x^\varepsilon \int_{-\infty}^{\infty}
\frac{{\rm d} t}{(|t|+1)^{1+\varepsilon}}
\\\noalign{\vskip 1mm}
& \ll_{\pi,\varepsilon} x^\varepsilon.
\end{align*}
This completes the proof.
\end{proof}

Now we are ready to complete the proof of Theorem \ref{thm1}. 

\medskip 
\noindent 
{\it Proof of Theorem~\ref{thm1}.}  
Define
$$
\lambda_{\pi}^{\pm}(n) := \frac{|\lambda_{\pi}(n)|\pm \lambda_{\pi}(n)}{2}\cdot
$$
It is easy to see that
$$
\lambda_{\pi}^{+}(n) = \begin{cases}
\lambda_{\pi}(n) & \text{if $\,\lambda_{\pi}^{+}(n)>0$,}
\\\noalign{\vskip 1mm}
0                          & \text{if $\,\lambda_{\pi}^{+}(n)<0$,}
\end{cases}
$$
and
$$
\lambda_{\pi}^{-}(n) = \begin{cases}
-\lambda_{\pi}(n) & \text{if $\,\lambda_{\pi}^{+}(n)<0$,}
\\\noalign{\vskip 1mm}
0                          & \text{if $\,\lambda_{\pi}^{+}(n)\ge 0$.}
\end{cases}
$$
It follows from \eqref{UBS1} and \eqref{LBS1*} that 
\begin{equation}\label{4.13}
\sum_{n\le x} \lambda_{\pi}^{\pm}(n) \Big(\log\frac{x}{n}\Big)^{[m/2]+1}
\gg x^{1-\theta_m} (\log x)^{1/m-1}
\end{equation}
for $x\ge x_0(\pi)$.

On the other hand, for any fixed $A>0$, by using \eqref{2ndMoment} of Theorem \ref{thm4} and by a simple partial integration, we find that
\begin{equation}\label{4.14}
\begin{aligned}
\sum_{n\le x} |\lambda_{\pi}(n)|^2 \Big(\log\frac{x}{n}\Big)^{A}
& = \int_{1-}^x \Big(\log\frac{x}{t}\Big)^{A} \d\Big(\sum_{n\le t} |\lambda_{\pi}(n)|^2\Big)
\\
& = A \int_1^x \sum_{n\le t} |\lambda_{\pi}(n)|^2 \Big(\log\frac{x}{t}\Big)^{A-1} \frac{\d t}{t}
\\
& \ll_{\pi} \int_1^x \Big(\log\frac{x}{t}\Big)^{A-1} \d t
\quad
(u = \log(x/t))
\\
& \ll_{\pi} x \int_0^{\log x} u^{A-1} \text{e}^{-u} \d u
\\\noalign{\vskip 2mm}
& \ll_{\pi} x.
\end{aligned}
\end{equation}
From \eqref{4.13} and \eqref{4.14},
we deduce, via the Cauchy-Schwarz inequality, that 
\begin{align*}
x^{1-\theta_m} (\log x)^{1/m-1}
& \ll_{\pi} \sum_{n\le x} \lambda_{\pi}^{\pm}(n) \Big(\log\frac{x}{n}\Big)^{[m/2]+1}
\\
& \ll_{\pi} \bigg\{\sum_{n\le x} |\lambda_{\pi}(n)|^2 \Big(\log\frac{x}{n}\Big)^{2[m/2]+2}
\sum_{\substack{n\le x\\ \lambda_{\pi}(n)\gtrless\,0}} 1\bigg\}^{1/2}
\\
& \ll_{\pi} 
\big\{x {\mathscr N}_{\pi}^{\pm}(x)\big\}^{1/2}.
\end{align*}
This proves Theorem~\ref{thm1}.  
\hfill $\square$

\vskip 2mm

\noindent{\bf Note added after publication}. 
The authors would like to thank Professor Chaohua Jia for pointing out 
the significance of his manuscript \cite{Jiachaohua2014}  
after the publication of present paper 
in the Journal of Number Theory.

\end{document}